\documentclass[11pt,a4paper]{amsart}
\usepackage[text={140mm,205mm},centering]{geometry}
\usepackage{amsfonts,amsthm,amsmath,amssymb}
\usepackage{amscd}
\usepackage{setspace}
\usepackage[loose,nice]{units} 
\usepackage{tensor}
\usepackage{mathrsfs}
\usepackage{kpfonts}
\usepackage{enumitem}
\usepackage{hyperref}

\DeclareSymbolFont{matha}{OML}{txmi}{m}{it}
\DeclareMathSymbol{\varv}{\mathord}{matha}{118}
\numberwithin{equation}{section}

\theoremstyle{plain}
\newtheorem{thm}{Theorem}[section]

\newtheorem{prop}[thm]{Proposition}
\newtheorem*{cor}{Corollary}

\theoremstyle{definition}
\newtheorem{defn}{Definition}[section]

\theoremstyle{remark}

\newtheorem*{note}{Note}

\newcommand{\interior}[1]{%
  {\kern0pt#1}^{\mathrm{o}}%
}

\begin{document}

\title{Real Projective Geodesics Embedded In Complex Manifolds}
\author[S. Sen]{Swagatam Sen}
\email{swagatam.sen@gmail.com}
\date{July 15, 2020}



\begin{abstract}
Focus of this study is to explore some aspects of mathematical foundations for using complex manifolds as a model for space-time. More specifically, certain equations of motions have been derived as a Projective geodesic on a real manifold embedded within a complex one. To that goal, first the geodesic on complex manifold has been computed using local complex and conjugate coordinates, and then its projection on the real sub-manifold has been studied. The projective geodesic, thus obtained, is shown to have additional terms beyond the usual Christoffel symbols, and hence expands the geodesic to capture effects beyond the mere gravitational ones.
\end{abstract}
\maketitle

\section{Introduction}
Motivations behind studying Complex manifolds within the realms of gravitational physics, or specifically quantum gravity, is a topic that has drawn significant attention since the days of formulation of General Relativity \cite{klein}\cite{Newman}\cite{Trautman}. Einstein himself had progressed to consider a complex hermitian metric to unify gravity and electromagnetism, albeit defined on real 4-dimensional coordinates\cite{Einstein}\cite{Einstein2}. Much before that similar proposals for a complex metric was put forward by Weyl\cite{weyl}\cite{weyl2} and Soh\cite{Soh} among others. As an approach, it doesn't provide adequately strong structure to conceptually integrate the real and imaginary parts of the metric. For example, under a generic continuous transformation both parts of the metric behave quite independently. That line of investigation has persisted over time\cite{straumann2000paulis}\cite{plebanski}\\\cite{Davidson_2012}\cite{bengtsson2018generalized}, albeit in a much reduced capacity. We have a fundamental weakness in these approaches which has restricted its successes, because of the metric being considered as function of exclusively real coordinates, or In other words, the framework used in these studies are that of a real manifold with a complex vector bundle structure around it. It doesn't draw in the stronger structure of complex metric defined on a complex manifold. 

In a way, that same weakness can be attributed to the standard mathematical approach developed later for Quantum Mechanics, wherein measurements are conceptualised as Operators on a complex Hilbert space, which in turn is defined on real phase space coordinates for a quantum system\cite{neumann2018mathematical}. Consequently the underlying mathematical structure for quantum mechanics developed through numerous groundbreaking work\cite{Hooft_2016}\cite{1981PhT....34k..69W}, while tremendously accurate, appears to be incomplete in addressing the longstanding foundational challenges - that of interpretation\cite{copenhagen2009}\cite{sep-qm-manyworlds}\cite{SCHLOSSHAUER2013222}, and that of measurement\cite{1985ZPhyB..59..223J}\cite{zurek2001decoherence}\cite{zeh1995decoherence} among others. 

It was primarily through the development of Twistor theory by Penrose\cite{penrose}, that the idea of a complex underlying space to describe the reality branched out into an area of sustained focus. Essentially within Twistor theory, `real' physical fields are represented as complex objects defined on complex projective space, called Twistor space\cite{1973PhR.....6..241P}. Mathematically, that approach offers a much richer structure to work with using contour integral formalism\cite{Mason_2009}\cite{Arkani_Hamed_2010}, and using connections with String theory\cite{skinner2013twistor}\cite{Witten_2004}. 

In the present work we adopt the view of a spacetime model that's inherently a complex manifold with a real Riemannian sub-manifold describing the measurable outcome. Specifically we look into the behaviour of Geodesics on a complex manifold and the corresponding trajectory it traces as a projection onto the real sub-manifold. 

Throughout the next section we would work with a complex n-manifold $\mathscr{M}$ and its {\it projective} real submanifold $\mathscr{R} \subset \mathscr{M}$. Our objective would be to derive the geodesic on the complex manifold first using the usual variational technique. Then we would study how that geodesic maps to a curve on the {\it projective sub-manifold} $\mathscr{R}$. We would refer this as the {\it projective real geodesic} on $\mathscr{R}$ and analyse how that differs from the true geodesic on $\mathscr{R}$.  

\begin{defn}
For a complex n-manifold $\mathscr{M}$ with a holomorphic atlas $\{(z,U)\}$, $\mathscr{R} \subset \mathscr{M}$ is a {\it projective real sub-manifold} if $\forall \text{ holomorphic charts }(z,U)\text{ on } \mathscr{M}, \exists \text{ a projection } \pi : \mathscr{M} \mapsto \mathscr{R} \text{ and a continuous chart } \chi : \pi(U) \mapsto \mathbb{R}^n  \text{ such that } \chi\circ \pi = Re(z)$ \qed
\end{defn}

This would allow us to write down the local coordinates as $z^\alpha = x^\alpha + it^\alpha$ and $\overline{z}^\beta = x^\beta - it^\beta$ where $x = \chi \circ \pi$ and $t = Im(z)$.

Next we would need to introduce a hermitian positive definite metric $g$ on $\mathscr{M}$ with the usual distance 2-form in local coordinates,
\begin{equation*}
d\tau^2 = g_{\alpha\overline{\beta}}dz^\alpha\overline{dz^\beta}
\end{equation*}

Our end goal in the next section would be to derive a geodesic equation using the complex/conjugate coordinates which is summarised in Theorem \ref{thm:1}. Then in Theorem \ref{thm:2}, we would further prove that this geodesic allows a {\it projective} geodesic on the real coordinates which has the generic structural form as below,

\begin{equation*}
h_{\mu\gamma}D^2x^\mu = \Upsilon^{(1,1)}_{\gamma\alpha\beta} Dx^\alpha Dx^\beta - \Upsilon^{(1,0)}_{\gamma\alpha\beta} Dx^\beta Dt^\alpha - \Upsilon^{(0,0)}_{\gamma\alpha\beta} Dt^\alpha Dt^\beta
\end{equation*}
for some tensor $h$ and symbols $\Upsilon^{(0,0)}, \Upsilon^{(1,0)}, \Upsilon^{(1,1)}$ all of which are functions of the metric.

These theorems would then allow us to consider certain known special cases and illustrate how equations of motions for gravitational and Lorentz Fields can naturally manifest from Theorem \ref{thm:2} under special configurations.

\section{Results}
To achieve the initial goal of deriving the geodesic on the complex manifold using complex/conjugate coordinates, we would follow the usual approach of variational calculus.

Under the metric $g$, the arc length for a given path parametrised by $\sigma \in [0 ,1]$, would be $\tau = \int\limits_0^1 L(\mathbf{z},\frac{d\mathbf{z}}{d\sigma}) d\sigma$, where 
\begin{equation*}
L = \sqrt{g_{\alpha\overline{\beta}}\frac{dz^\alpha}{d\sigma}\frac{d\overline{z^\beta}}{d\sigma}}
\end{equation*}
is the Lagrangian for equation of motion.
\begin{note}
Since $g$ is hermitian, $L$ is real and we can think of minimising the arc length in the usual variational approach.\qed
\end{note}
\begin{note}
the path can be re-parametrized using $\tau$ instead of arbitrary $\sigma$ and the the derivatives can be replaced as $\frac{d}{d\sigma} = L\frac{d}{d\tau}$. Notationally henceforth we would use $D = \frac{d}{d\tau}$. \qed
\end{note}

\begin{prop}\label{prop:1}
Let $Dz^\gamma = \frac{dz^\gamma}{d\tau}$ and $D\overline{z}^\gamma = \frac{d\overline{z}^\gamma}{d\tau}$ be path derivatives of the complex and conjugate coordinates. Also let $\partial_\alpha = \frac{\partial}{\partial z^\alpha}$ and $\partial_{\overline{\beta}} = \frac{\partial}{\partial \overline{z^\beta}}$ be the partial derivatives with respect to complex and conjugate coordinates. Then,
\begin{equation*}
\frac{d}{d\sigma}\big[\frac{\partial L}{\partial (d\overline{z^\gamma}/d\sigma)}\big] = \frac{L}{2}\big[\partial_\beta g_{\alpha\overline{\gamma}} Dz^\alpha Dz^\beta + \partial_{\overline{\beta}} g_{\alpha\overline{\gamma}} Dz^\alpha D\overline{z}^\beta + g_{\alpha\overline{\gamma}} D^2 z^\alpha\big]
\end{equation*}
\qed
\end{prop}
\begin{proof}
From definition of $L$ we know that,
\begin{equation*}
\frac{\partial L}{\partial (d\overline{z^\gamma}/d\sigma)} = \frac{1}{2L} g_{\alpha\overline{\gamma}} \frac{dz^\alpha}{d\sigma}
\end{equation*}

That, in turn, would allow us to write,
\begin{equation*}
\begin{aligned}
\frac{d}{d\sigma}\frac{\partial L}{\partial (d\overline{z^\gamma}/d\sigma)} &= \frac{d}{d\sigma}\frac{1}{2L} g_{\alpha\overline{\gamma}} \frac{dz^\alpha}{d\sigma}
			= \frac{L}{2} \frac{d}{d\tau} g_{\alpha\overline{\gamma}} \frac{dz^\alpha}{d\tau}\\
			&= \frac{L}{2} \big[\frac{d g_{\alpha\overline{\gamma}}}{d\tau} \frac{dz^\alpha}{d\tau} + g_{\alpha\overline{\gamma}} \frac{d^2 z^\alpha}{d\tau^2}\big]\\
			&= \frac{L}{2} \big[\partial_\beta g_{\alpha\overline{\gamma}} \frac{dz^\alpha}{d\tau}\frac{dz^\beta}{d\tau} +\partial_{\overline{\beta}} g_{\alpha\overline{\gamma}}\frac{dz^\alpha}{d\tau}\frac{d\overline{z^\beta}}{d\tau} + g_{\alpha\overline{\gamma}} \frac{d^2 z^\alpha}{d\tau^2}\big]
\end{aligned}
\end{equation*}

\end{proof}

We can now use this result to derive the geodesic equation in complex and conjugate coordinates.

\begin{thm}\label{thm:1}
Let $\mathscr{M}$ be a complex manifold with a hermitian metric $g$. If $\tau$ is the cumulative length of a given path $\Omega$ on $\mathscr{M}$, then $\Omega$ is a geodesic iff 
\begin{equation*}
g_{\mu\overline{\gamma}} D^2z^\mu = \big(\partial_{\overline{\gamma}} g_{\alpha\overline{\beta}} - \partial_{\overline{\beta}} g_{\alpha\overline{\gamma}}\big)Dz^\alpha D\overline{z^\beta} - \partial_\alpha g_{\beta\overline{\gamma}} Dz^\alpha Dz^\beta
\end{equation*}
\qed
\end{thm}
\begin{proof}
We'd first note that,
\begin{equation*}
\frac{\partial L}{\partial \overline{z^\gamma}} = \frac{L}{2} \partial_{\overline{\gamma}} g_{\alpha\overline{\beta}} Dz^\alpha D\overline{z^\beta}
\end{equation*}

Now the result follows from a direct application of the Euler-Lagrange Eqn on conjugate coordinates 
\begin{equation*}
\frac{\partial L}{\partial \overline{z^\gamma}} = \frac{d}{d\sigma}\big[\frac{\partial L}{\partial (d\overline{z^\gamma}/d\sigma)}\big]
\end{equation*}
and replacing the right hand side of the equation using Proposition \ref{prop:1}.
\end{proof}

We can use Theorem \ref{thm:1}, to understand the trace of the complex geodesic projected on the real coordinates and how different that trajectory is from the geodesic on the real sub-manifold. But before that we'd need to introduce a set of objects that would serve the equivalent purpose of Christoffel symbols in the projective geodesic.
\begin{defn}
For a hermitian, positive definite metric $g$ on a Complex manifold $\mathscr{M}$, we define {\it Primary Field} as the combination 
\begin{equation*}
\Phi = \big(\Phi^{\gamma++}_{\alpha\beta}, \Phi^{\gamma+-}_{\alpha\beta}, \Phi^{\gamma-+}_{\alpha\beta}, \Phi^{\gamma--}_{\alpha\beta}\big)
\end{equation*}
where
\begin{itemize}
\item{$\Phi^{\gamma++}_{\alpha\beta} = \frac{1}{2}\big[\partial^x_\gamma g^R_{\alpha\overline{\beta}} - \partial^x_\alpha g^R_{\beta\overline{\gamma}} - \partial^x_\beta g^R_{\alpha\overline{\gamma}}\big]$}
\item{$\Phi^{\gamma+-}_{\alpha\beta} = \frac{1}{2}\big[\partial^x_\gamma g^I_{\alpha\overline{\beta}} - \partial^x_\alpha g^I_{\beta\overline{\gamma}} - \partial^x_\beta g^I_{\alpha\overline{\gamma}}\big]$}
\item{$\Phi^{\gamma-+}_{\alpha\beta} = \frac{1}{2}\big[\partial^t_\gamma g^R_{\alpha\overline{\beta}} - \partial^t_\alpha g^R_{\beta\overline{\gamma}} - \partial^t_\beta g^R_{\alpha\overline{\gamma}}\big]$}
\item{$\Phi^{\gamma--}_{\alpha\beta} = \frac{1}{2}\big[\partial^t_\gamma g^I_{\alpha\overline{\beta}} - \partial^t_\alpha g^I_{\beta\overline{\gamma}} - \partial^t_\beta g^I_{\alpha\overline{\gamma}}\big]$}
\end{itemize}
are the ordinary Christoffel symbols for real and imaginary parts of the metric in real and imaginary coordinates respectively. \qed
\end{defn}

\begin{note}
If $g$ is real symmetric, then $\Phi^{\gamma+-}_{\alpha\beta} = \Phi^{\gamma--}_{\alpha\beta} = 0$
\end{note}
\begin{note}
If $g$ is independent of $t$, then $\Phi^{\gamma-+}_{\alpha\beta} = \Phi^{\gamma--}_{\alpha\beta} = 0$
\end{note}
\begin{note}
$\Phi^{\gamma+-}_{\alpha\alpha} = \Phi^{\gamma--}_{\alpha\alpha} = 0$
\end{note}

\begin{prop}\label{prop:2}
For a hermitian metric $g=g^R+ig^I$ with {\it Primary Field} $\Phi$, we can write,
\begin{enumerate}
\item{$\partial_{\overline{\gamma}} g_{\alpha\overline{\beta}} - \partial_{\overline{\beta}} g_{\alpha\overline{\gamma}} - \partial_{\alpha} g_{\beta\overline{\gamma}} = \Phi^{\gamma++}_{\alpha\overline{\beta}} + i\Phi^{\gamma+-}_{\alpha\overline{\beta}} +\frac{i}{2}\partial^t_\gamma g^R_{\alpha\overline{\beta}}$}
\item{$\partial_{\overline{\gamma}} g_{\alpha\overline{\beta}} - \partial_{\overline{\beta}} g_{\alpha\overline{\gamma}} + \partial_{\alpha} g_{\beta\overline{\gamma}} = \frac{1}{2}\partial^x_\gamma g^R_{\alpha\overline{\beta}} - \Phi^{\gamma;--}_{\alpha\overline{\beta}} + i\Phi^{\gamma-+}_{\alpha\overline{\beta}}$}
\end{enumerate}
\end{prop}

\begin{proof}
As we know,,
\begin{equation*}
\begin{aligned}
\partial_{\overline{\gamma}} g_{\alpha\overline{\beta}} - \partial_{\overline{\beta}} g_{\alpha\overline{\gamma}} = \frac{1}{2}\big[\partial^x_\gamma g^R_{\alpha\overline{\beta}} - \partial^t_\gamma g^I_{\alpha\overline{\beta}} &- \partial^x_\beta g^R_{\alpha\overline{\gamma}} + \partial^t_\beta g^I_{\alpha\overline{\gamma}}\big] \\&+\frac{i}{2}\big[\partial^x_\gamma g^I_{\alpha\overline{\beta}} + \partial^t_\gamma g^R_{\alpha\overline{\beta}} - \partial^x_\beta g^I_{\alpha\overline{\gamma}} - \partial^t_\beta g^R_{\alpha\overline{\gamma}}\big]
\end{aligned}
\end{equation*}

we can write,
\begin{equation*}
\begin{aligned}
\partial_{\overline{\gamma}} g_{\alpha\overline{\beta}} - \partial_{\overline{\beta}} g_{\alpha\overline{\gamma}} - \partial_{\alpha} g_{\beta\overline{\gamma}} &= \frac{1}{2}\big[\partial^x_\gamma g^R_{\alpha\overline{\beta}} - \partial^x_\beta g^R_{\alpha\overline{\gamma}} -\partial^x_\alpha g^R_{\beta\overline{\gamma}} \big] \\&+ \frac{i}{2} \big[ \partial^x_\gamma g^I_{\alpha\overline{\beta}} -\partial^x_\beta g^I_{\alpha\overline{\gamma}} - \partial^x_\alpha g^I_{\beta\overline{\gamma}}\big] +\frac{i}{2}\partial^t_\gamma g^R_{\alpha\overline{\beta}}\\
&=\Phi^{\gamma++}_{\alpha\overline{\beta}} + i\Phi^{\gamma+-}_{\alpha\overline{\beta}} +\frac{i}{2} \partial^t_\gamma g^R_{\alpha\overline{\beta}}
\end{aligned}
\end{equation*}

Similarly,
\begin{equation*}
\begin{aligned}
\partial_{\overline{\gamma}} g_{\alpha\overline{\beta}} - \partial_{\overline{\beta}} g_{\alpha\overline{\gamma}} + \partial_{\alpha} g_{\beta\overline{\gamma}} &= \frac{1}{2}\partial^x_\gamma g^R_{\alpha\overline{\beta}} - \frac{1}{2}\big[\partial^t_\gamma g^I_{\alpha\overline{\beta}} - \partial^t_\beta g^I_{\alpha\overline{\gamma}} -\partial^t_\alpha g^I_{\beta\overline{\gamma}}\big] \\
&+ \frac{i}{2} \big[ \partial^t_\gamma g^R_{\alpha\overline{\beta}} -\partial^t_\beta g^R_{\alpha\overline{\gamma}} - \partial^t_\alpha g^R_{\beta\overline{\gamma}}\big] \\
&=-\Phi^{\gamma--}_{\alpha\overline{\beta}} + i\Phi^{\gamma-+}_{\alpha\overline{\beta}} +\frac{1}{2} \partial^x_\gamma g^R_{\alpha\overline{\beta}}
\end{aligned}
\end{equation*}

\end{proof}

\begin{defn}
For a hermitian, positive definite metric $g$ on a Complex manifold $\mathscr{M}$, we define {\it Secondary Field} as 
\begin{equation*}
F=\big(F^{x}_{\alpha\gamma\beta},F^{t}_{\alpha\gamma\beta}\big)
\end{equation*}

where

\begin{itemize}
\item{$F^{x}_{\alpha\gamma\beta} = \partial^x_\gamma g^I_{\alpha\overline{\beta}} - \partial^x_\beta g^I_{\alpha\overline{\gamma}}$}
\item{$F^{t}_{\alpha\gamma\beta} = \partial^t_\gamma g^I_{\alpha\overline{\beta}} - \partial^t_\beta g^I_{\alpha\overline{\gamma}}$}
\end{itemize} \qed
\end{defn}

\begin{note}
{\it Secondary Field} symbols are anti-symmetric in $\beta, \gamma$
\end{note}
\begin{note}
If $g$ is real symmetric, then $F^{x}_{\alpha\gamma\beta} = F^{t}_{\alpha\gamma\beta} = 0$
\end{note}
\begin{note}
{\it Secondary Field} symbols are covariant 2-tensors in $\beta, \gamma$ indices under holomorphic coordinate changes.
\end{note}

\begin{prop}\label{prop:3}
For a hermitian metric $g=g^R+ig^I$ with {\it Secondary field} $F$, we can write,
\begin{equation*}
2\partial_{\overline{\gamma}} g^I_{\alpha\overline{\beta}} + i\partial^x_\beta g_{\alpha\overline{\gamma}} + \partial^t_\alpha g_{\beta\overline{\gamma}} = F^x_{\alpha\gamma\beta} + \partial^t_\alpha g^I_{\beta\overline{\gamma}} - iF^t_{\beta\gamma\alpha} + i\partial^x_\beta g^R_{\alpha\overline{\gamma}}
\end{equation*}
\end{prop}

\begin{proof}
Clearly, 
\begin{equation*}
\begin{aligned}
2\partial_{\overline{\gamma}} g^I_{\alpha\overline{\beta}} &+ i\partial^x_\beta g_{\alpha\overline{\gamma}} + \partial^t_\alpha g_{\beta\overline{\gamma}} \\
&= \partial^x_\gamma g^I_{\alpha\overline{\beta}} + i\partial^t_\gamma g^I_{\alpha\overline{\beta}} + i\partial^x_\beta g^R_{\alpha\overline{\beta}} - \partial^x_\beta g^I_{\alpha\overline{\gamma}} + \partial^t_\alpha g^R_{\beta\overline{\gamma}}  + i\partial^t_\alpha g^I_{\beta\overline{\gamma}} \\
&= (\partial^x_\gamma g^I_{\alpha\overline{\beta}} - \partial^x_\beta g^I_{\alpha\overline{\gamma}}) + \partial^t_\alpha g^I_{\beta\overline{\gamma}} -i(\partial^t_\gamma g^I_{\alpha\overline{\beta}} - \partial^t_\alpha g^I_{\beta\overline{\gamma}}) + i\partial^x_\beta g^R_{\alpha\overline{\gamma}} \\
&=F^x_{\alpha\gamma\beta} + \partial^t_\alpha g^I_{\beta\overline{\gamma}} - iF^t_{\beta\gamma\alpha} + i\partial^x_\beta g^R_{\alpha\overline{\gamma}}
\end{aligned}
\end{equation*}
\end{proof}
\begin{defn}
For a hermitian metric $g$ on a Complex manifold $\mathscr{M}$ with $g = g^R + ig^I$, we define Link tensor as $\epsilon^\eta_\gamma = g^I_{\nu\overline{\gamma}}g^{R;\nu\overline{\eta}}$ \qed
\end{defn}

\begin{note}
If $g$ is real symmetric then $\epsilon^\eta_\gamma = 0$
\end{note}

\begin{thm}\label{thm:2}
Let $\mathscr{R}$ be a real projective manifold embedded within a complex manifold $\mathscr{M}$. Let $g$ be a hermitian, positive definite metric on $\mathscr{M}$ with $g = g^R+ig^I$. Additionally let $\Phi, F$ be the {\it primary and secondary fields} respectively and let $\epsilon$ be the link symbol. Then $\mathscr{R}$ is endowed with a {\it real projective geodesic} described by
\begin{equation*}
\big(g^R_{\mu\overline{\gamma}} + \epsilon^\eta_\gamma g^I_{\mu\overline{\eta}}\big) D^2x^\mu = \Upsilon^{(1,1)}_{\gamma\alpha\beta} Dx^\alpha Dx^\beta - \Upsilon^{(1,0)}_{\gamma\alpha\beta} Dx^\beta Dt^\alpha - \Upsilon^{(0,0)}_{\gamma\alpha\beta} Dt^\alpha Dt^\beta
\end{equation*}
where 
\begin{itemize}
\item{$\Upsilon^{(1,1)}_{\gamma\alpha\beta} = \Phi^{\gamma++}_{\alpha\overline{\beta}} + \epsilon^\eta_\gamma \Phi^{\gamma-+}_{\alpha\overline{\beta}} + \frac{1}{2}\epsilon^\eta_\gamma \partial^t_\eta g^R_{\alpha\overline{\beta}}$}
\item{$\Upsilon^{(1,0)}_{\gamma\alpha\beta} = F^{x}_{\alpha\gamma\beta} - \epsilon^\eta_\gamma F^{t}_{\beta\eta\alpha} +\partial^t_\alpha g^R_{\beta\overline{\gamma}} + \epsilon^\eta_\gamma \partial^x_\beta g^R_{\alpha\overline{\eta}}$}
\item{$\Upsilon^{(0,0)}_{\gamma\alpha\beta} = \Phi^{\gamma--}_{\alpha\overline{\beta}} - \frac{1}{2} \partial^x_\gamma g^R_{\alpha\overline{\beta}} - \epsilon^\eta_\gamma \Phi^{\gamma-+}_{\gamma\alpha\overline{\beta}}$}
\end{itemize}
\end{thm} \qed

\begin{proof}

First of all we can write,
\begin{equation*}
\begin{aligned}
\big[\partial_{\overline{\gamma}} &g_{\alpha\overline{\beta}} - \partial_{\overline{\beta}} g_{\alpha\overline{\gamma}}\big] Dz^\alpha D\overline{z}^\beta \\&=  \big[\partial_{\overline{\gamma}} g_{\alpha\overline{\beta}} - \partial_{\overline{\beta}} g_{\alpha\overline{\gamma}}\big] \big[Dx^\alpha Dx^\beta + Dt^\alpha Dt^\beta +i\big(Dt^\alpha Dx^\beta - Dt^\beta Dx^\alpha\big)\big]\\
&=\big[\partial_{\overline{\gamma}} g_{\alpha\overline{\beta}} - \partial_{\overline{\beta}} g_{\alpha\overline{\gamma}}\big] \big[Dx^\alpha Dx^\beta + Dt^\alpha Dt^\beta\big] 
+i\big[2\partial_{\overline{\gamma}} g^I_{\alpha\overline{\beta}} - \partial_{\overline{\beta}} g_{\alpha\overline{\gamma}} + \partial_{\overline{\alpha}} g_{\beta\overline{\gamma}}\big] Dt^\alpha Dx^\beta
\end{aligned}
\end{equation*}

Also,
\begin{equation*}
\begin{aligned}
\partial_{\beta} g_{\alpha\overline{\gamma}} Dz^\alpha Dz^\beta &=   \partial_{\beta} g_{\alpha\overline{\gamma}} \big[Dx^\alpha Dx^\beta - Dt^\alpha Dt^\beta +i\big(Dt^\alpha Dx^\beta + Dt^\beta Dx^\alpha\big)\big]\\
&= \partial_{\alpha} g_{\beta\overline{\gamma}} \big[Dx^\alpha Dx^\beta - Dt^\alpha Dt^\beta\big] +i\big[\partial_{\beta} g_{\alpha\overline{\gamma}} + \partial_{\alpha} g_{\beta\overline{\gamma}}\big]Dt^\alpha Dx^\beta
\end{aligned}
\end{equation*}

That implies,
\begin{equation*}
\begin{aligned}
\big[\partial_{\overline{\gamma}} &g_{\alpha\overline{\beta}} - \partial_{\overline{\beta}} g_{\alpha\overline{\gamma}}\big] Dz^\alpha D\overline{z}^\beta - \partial_{\beta} g_{\alpha\overline{\gamma}} Dz^\alpha Dz^\beta \\&= \big[\partial_{\overline{\gamma}} g_{\alpha\overline{\beta}} - \partial_{\overline{\beta}} g_{\alpha\overline{\gamma}}-\partial_{\alpha} g_{\beta\overline{\gamma}}\big]Dx^\alpha Dx^\beta +\big[\partial_{\overline{\gamma}} g_{\alpha\overline{\beta}} - \partial_{\overline{\beta}} g_{\alpha\overline{\gamma}} + \partial_{\alpha} g_{\beta\overline{\gamma}}\big]Dt^\alpha Dt^\beta  \\
&-\big[2\partial_{\overline{\gamma}} g^I_{\alpha\overline{\beta}} + i\partial^x_\beta g_{\alpha\overline{\gamma}} + \partial^t_\alpha g_{\beta\overline{\gamma}}\big] Dt^\alpha Dx^\beta
\end{aligned}
\end{equation*}

Now using this last equation alongside Theorem \ref{thm:1}, Propositions \ref{prop:2} and \ref{prop:3}, we get,
\begin{equation*}
\begin{aligned}
g_{\mu\overline{\gamma}}D^2z^\mu &= \big[\partial_{\overline{\gamma}} g_{\alpha\overline{\beta}} - \partial_{\overline{\beta}} g_{\alpha\overline{\gamma}}\big] Dz^\alpha D\overline{z}^\beta - \partial_{\beta} g_{\alpha\overline{\gamma}} Dz^\alpha Dz^\beta \\ &= (\Phi^{\gamma++}_{\alpha\overline{\beta}} + i\Phi^{\gamma+-}_{\alpha\overline{\beta}} +\frac{i}{2}\partial^t_\gamma g^R_{\alpha\overline{\beta}})Dx^\alpha Dx^\beta + (\frac{1}{2}\partial^x_\gamma g^R_{\alpha\overline{\beta}} - \Phi^{\gamma;--}_{\alpha\overline{\beta}} + i\Phi^{\gamma-+}_{\alpha\overline{\beta}})Dt^\alpha Dt^\beta \\
&-(F^x_{\alpha\gamma\beta} + \partial^t_\alpha g^I_{\beta\overline{\gamma}} - iF^t_{\beta\gamma\alpha} + i\partial^x_\beta g^R_{\alpha\overline{\gamma}})Dt^\alpha Dx^\beta
\end{aligned}
\end{equation*}

But we'd note that,
\begin{equation*}
g_{\mu\overline{\gamma}}D^2z^\mu = \big(g^R_{\mu\overline{\gamma}}D^2x^\mu - g^I_{\mu\overline{\gamma}} D^2t^\mu\big) + i\big(g^R_{\mu\overline{\gamma}}D^2t^\mu + g^I_{\mu\overline{\gamma}}D^2x^\mu\big)
\end{equation*}

Equating real and imaginary parts we get the dual identity,

\begin{equation*}
\begin{aligned}
g^R_{\mu\overline{\gamma}}D^2x^\mu - g^I_{\mu\overline{\gamma}} D^2t^\mu = \Phi^{\gamma++}_{\alpha\overline{\beta}} Dx^\alpha Dx^\beta + (\frac{1}{2}\partial^x_\gamma g^R_{\alpha\overline{\beta}} &- \Phi^{\gamma;--}_{\alpha\overline{\beta}})Dt^\alpha Dt^\beta \\
&-(F^x_{\alpha\gamma\beta} + \partial^t_\alpha g^I_{\beta\overline{\gamma}} )Dt^\alpha Dx^\beta
\end{aligned}
\end{equation*}

and

\begin{equation*}
\begin{aligned}
g^R_{\mu\overline{\eta}}D^2t^\mu + g^I_{\mu\overline{\eta}}D^2x^\mu = (\Phi^{\eta+-}_{\alpha\overline{\beta}} +\frac{1}{2}\partial^t_\eta g^R_{\alpha\overline{\beta}})Dx^\alpha Dx^\beta &+ \Phi^{\eta-+}_{\alpha\overline{\beta}}Dt^\alpha Dt^\beta \\
&-(\partial^x_\beta g^R_{\alpha\overline{\eta}}- F^t_{\beta\eta\alpha})Dt^\alpha Dx^\beta
\end{aligned}
\end{equation*}

Multiplying the second identity by $\epsilon^\eta_\gamma$ and adding the identities, we arrive at the desired result.
\end{proof}
Next we'd investigate some of the specific ramifications and special cases of Theorem \ref{thm:2} through a set of corollaries.

\begin{cor}
Let $\mathscr{M}$ be a 4-dimensional complex manifold with a real symmetric metric $g$ and let $\mathscr{R}$ be a {\it projective real} sub-manifold. Additionally let's assume that $\forall \alpha,\beta$,
\begin{equation*}
\partial^x_1 g_{\alpha\overline{\beta}} = \partial^t_2 g_{\alpha\overline{\beta}} = \partial^t_3 g_{\alpha\overline{\beta}} = \partial^t_4 g_{\alpha\overline{\beta}} = 0
\end{equation*}
Then the projective geodesic is identical to the geodesic on $\mathscr{R}$.
\qed
\end{cor} 
\begin{proof}
Proof follows by considering $g^I=0$ and setting the specific partial derivatives to 0, in which case Theorem \ref{thm:2} simplifies to 
\begin{equation*}
\begin{aligned}
D^2x^\mu &= g^{\mu\gamma}\Phi^{\gamma++}_{\alpha\beta} Dx^\alpha Dx^\beta - g^{\mu\gamma}\partial_1^t g_{\beta\gamma} Dt^1 Dx^\beta - \frac{1}{2} g^{\mu\gamma}\partial^x_\gamma g_{11}(Dt^1)^2 
\end{aligned}
\end{equation*}
where $\alpha,\beta,\mu > 1$
This can be rewritten as
\begin{equation*}
D^2y^\mu = - \Gamma^\mu_{ab} Dy^a Dy^b
\end{equation*}
where $a,b>0,\mu>1$ and $y=(t^1,x^2,x^3,x^4)$. $\Gamma$ denotes the usual Christoffel symbols and the equation precisely signifies the classical geodesic path.
\end{proof}

We'd refer the quantity $G^\mu= - \Gamma^\mu_{ab} Dy^a Dy^b$ as the {\it Gravitation} field. 
\\
\\
It's interesting to note that Theorem \ref{thm:2} doesn't prescribe a single equation for $D^2x^\mu$. Rather it describes a family of geodesics parametrised by the extrinsic parameters $Dt^\alpha$. One particular choice of such parameters and consequentially, geodesics, would be the Root Geodesics, where $Dt^\alpha = 1, \forall \alpha$

\begin{cor}
Let $\mathscr{M}$ be a 4-dimensional complex manifold with a hermitian metric $g$ with $\epsilon^\eta_\gamma \rightarrow 0$. Also let $\mathscr{R}$ be a {\it projective real} sub-manifold. Additionally let's assume that $\forall \alpha,\beta$,
\begin{equation*}
\partial^x_1 g_{\alpha\overline{\beta}} = \partial^t_2 g_{\alpha\overline{\beta}} = \partial^t_3 g_{\alpha\overline{\beta}} = \partial^t_4 g_{\alpha\overline{\beta}} = 0
\end{equation*}
Then the projective root geodesic is given by,
\begin{equation*}
D^2x^\mu = G^\mu + L^\mu
\end{equation*}
where 
\begin{equation*}
L^\mu = -g^{R;\mu\overline{\gamma}}F^x_{1\gamma\beta}Dx^\beta - g^{R;\mu\overline{\gamma}}\Phi^{\gamma--}_{1\overline{1}}
\end{equation*}
\qed
\end{cor}
\begin{proof}
Proof follows by admitting the limit $\epsilon^\eta_\gamma \rightarrow 0$ and allowing $Dt^1=1$ in Theorem \ref{thm:2}
\begin{equation*}
\begin{aligned}
g^R_{\mu\overline{\gamma}}D^2x^\mu &= \Phi^{\gamma++}_{\alpha\beta} Dx^\alpha Dx^\beta - \partial_1^t g_{\beta\gamma} Dx^\beta - \frac{1}{2} g^{\mu\gamma}\partial^x_\gamma g_{11}\\
&- F^x_{1\gamma\beta} Dx^\beta - \Phi^{\gamma--}_{1\overline{1}}\\
&= g^R_{\mu\overline{\gamma}} G^\gamma - F^x_{1\gamma\beta} Dx^\beta - \Phi^{\gamma--}_{1\overline{1}}
\end{aligned}
\end{equation*}
\end{proof}
$L^\mu$ would be referred as the Lorentz field. Indeed when $g^R_{\mu\overline{\gamma}}=\delta_{\mu\gamma}$ is Euclidean, we have a familiar form of a Lorenz field,
\begin{equation*}
L^\mu = -F^x_{1\gamma\beta}Dx^\beta - \Phi^{\gamma--}_{1\overline{1}}
\end{equation*}
where $F^x_{1\gamma\beta}$ represents anti-symmetric Magnetic tensor and $\Phi^{\gamma--}_{1\overline{1}}$ represents Electric field. In fact, since,
\begin{equation*}
F^{x}_{1\gamma\beta} = \partial^x_\gamma g^I_{1\overline{\beta}} - \partial^x_\beta g^I_{1\overline{\gamma}}
\end{equation*}
naturally we can characterise $A_\gamma = g^I_{1\overline{\gamma}}$ as the magnetic potential. If we allow $\phi$ such that $\partial^x_\gamma \phi = \Phi^{\gamma--}_{11}$ then we can write $(\phi, A_2, A_3, A_4)$ as the 4-potential for the field. 

In the most general case when $\epsilon^\eta_\gamma$ isn't insignificant but still $\|\epsilon\| < 1$, we can rewrite Theorem \ref{thm:2} for $D^2x^\mu$ as an infinite sum of perturbations involving diminishing powers of $\epsilon$.

\section{Discussion and Conclusion}
In the preceding section we have been able to derive the most generic form for the {\it Projective real} geodesics. While Theorem \ref{thm:1} describes the said geodesic over complex and conjugate coordinates, in Theorem \ref{thm:2} we've been able to arrive at a much more useful expression for the geodesic in terms of real and imaginary coordinates. Through the results, we have seen that in general the derived projective geodesic deviates from the true geodesic on the real coordinates, and indeed contains additional contributions to the `Force field'. Careful examination of these additional terms reveal that they vanish when the underlying metric is real symmetric and the manifold itself can be locally mapped to the Minkowski spacetime. For these {\it Gravitational systems} it is shown that the Projective geodesic is identical to the true geodesic prescribed by General Relativity. 
\\
\\
However, in presence of an anti-symmetric imaginary component of the metric, we have seen that Theorem \ref{thm:2} naturally yields additional terms. When the imaginary component of the metric is relatively small in magnitude compared to the real part, these additional terms are shown to manifest as familiar Lorentz field with 4-potential being a function of the metric itself. This is a remarkable result because unlike the existing relativistic approaches to Electromagnetism, it doesn't merely show that Lorenz Field is compatible with GR, but rather it is a manifestation of the spacetime itself, as is gravity. 
\\
\\
For the case when both real and imaginary parts of the metric are of comparable scale, the Projective geodesic incrementally covers an infinite series of perturbative corrections involving diminishing contributions from the higher powers of the link tensor. This general case and perturbative terms would need to be investigated in more details within a separate follow up study.

\bibliography{Real-Projective-Geodesics.bib}
\bibliographystyle{unsrt}
 \end{document}